\documentclass[11.23pt, reqno]{amsart}
\usepackage{amsmath,amssymb}
\usepackage{amsthm,amstext}
\usepackage{amsmath}
\usepackage{graphicx}
\usepackage{amsfonts}
\usepackage{amssymb}
\usepackage[colorlinks,linkcolor=blue,citecolor=blue]{hyperref}

\newtheorem{thm}{Theorem}[section]

\newtheorem{corollary}[thm]{Corollary}
\newtheorem{lemma}[thm]{Lemma}
\newtheorem{prop}[thm]{Proposition}

\theoremstyle{definition}
\newtheorem{defin}[thm]{Definition}

\newtheorem{remark}{ Remark}

\newtheorem{example}{ Example}

\numberwithin{equation}{section}

\date{}

\begin{document}
 \baselineskip=17pt
\title{Remarks on  GP-metric and  partial metric spaces  \\and fixed points results}
\author{  M. R. Ahmadi Zand }
\address{  Mathematics Department, Yazd university}
\email{mahmadi@yazd.ac.ir}

\author{ H. Golvardi Yazdy}
\address{  Mathematics Department, Yazd university}
\email{ h-golvardy@yahoo.com}
\date{}
 \begin{abstract}
\baselineskip=10pt
In the present paper, we show that every fixed point theorem  on G-metric (respectively partial metric) spaces implies a fixed point result  on partial metric (respectively GP-metric)   spaces. 
 Our results generalized and unify several fixed point theorems in the context of partial metric spaces and GP-metric spaces in the literature.
   \end{abstract}
\maketitle
\section{ Introduction and basic observations}\label{secintro}

Matthews, a computer scientists who is an expert on semantics, introduced the notion of partial metric spaces   in 1992 \cite {st}. Then, this concept was studied by many authors, for example see \cite{M.,IA1,IAFS,ps,bkmp,Lg,kar,5, IO}. \par
Let $X$ be a non-empty set. A {\it partial metric}  is a function $p:X\times X \rightarrow [0, \infty )$ such that,
\begin{enumerate}
\item[(P1)]
$x=y$ if and only if $p(x, x) = p(x, y) = p(y, y)$, for all $x, y\in X$,
\item[(P2)]
$p(x, x)\leq p(x, y)$, for all $x, y\in X$ (small self-distances),
\item[(P3)]
$p(x, y)= p(y, x)$, for all $x, y\in X$ (symmetry),
\item[(P4)]
$p(x, z)\leq p(x, y)+ p(y, z) -p(y, y)$, for all $x, y, z\in X$ (triangularity).
\end{enumerate}
 The pair $(X, p)$ is called a {\it partial metric } space.
 Also, each partial metric $p$ on X generates a $T_0$ topology $\tau_p$
on X with a base of the family of open $p$-balls $\{ Bp(x, \epsilon): x \in  X, \epsilon > 0 \}$, where $Bp(x, \epsilon) = \{y \in  X: p(x, y) < p(x, x) + \epsilon \}$
\cite{bkmp,st,9}. If $p$ is a partial metric on X, then $p^{s} : X \times X \to [0,\infty)$ defined by $p^{s}(x, y) = 2p(x, y) - p(x, x) - p(y, y) $ is a
metric on X. 
 \par Let $(X, p)$ be
a partial metric space. A sequence $ \{x_{n} \}_{n\geq 1}$ in X converges to $x \in  X$  if  $\lim_{n\to \infty} p(x_{n}, x) = p(x, x)$ \cite{st}. Also, $ \{x_{n} \}_{n\geq 1}$ is called a Cauchy sequence whenever $\lim_{n,m \to \infty} p(x_{n}, x_{m})$ exists and is finite.
 A partial metric space $(X, p)$ is
said to be complete whenever every Cauchy sequence $ \{x_{n} \}_{n\geq 1}$ in X converges to a point $x \in  X$ with respect to $\tau_p$ , that is,
$\lim_{n,m \to \infty} p(x_{n}, x_{m})=p(x,x)$.
\begin{lemma} \text{(see \cite{9})}\label{1} Let $(X, p)$ be a partial metric space. Then,
\begin{itemize}
\item[(1)] $ \{x_{n} \}_{n\geq 1}$ is a Cauchy sequence in $(X, p)$ if and only if $ \{x_{n} \}_{n\geq 1}$ is a Cauchy sequence in the metric space $(X, p^{s})$,
\item[(2)] $(X, p)$ is complete if and only if the metric space $(X, p^{s})$ is complete.
\end{itemize}
\end{lemma}
\section{Remarks on  G-metric spaces and fixed points theorems}

\par Mustafa and Sims \cite {Zead} introduced a new structure of generalized metric spaces which are called G-metric spaces to develop and introduce a new fixed point theory for various mappings in this new structure. In this section, a new  connection between partial metric and G-metric spaces are established   and then  we show that any  fixed point theorem on G-metric spaces implies a  fixed point result on  partial metric spaces. First, we  recall some definitions of G-metric spaces and some properties of theirs. Let $X$ be a non-empty set and  $G:X \times X
 \times X\longrightarrow  [0, \infty )$  be a function satisfying  the following conditions:
 \begin{itemize}
 \item[(G1)]
  if $x = y= z$, then $G(x,y,z)=0$,
  \item[(G2)]
  $G(x,x,y)>0$  for all  $x,y\in X$ with  $x\neq y$,
  \item[(G3)]
   $G(x,x,y) \leq G(x,y,z)$ for all  $x,y,z\in X$ with  $z\neq y$,
   \item[(G4)]
   $G(x,y,z)=G( p \{x,z,y\}) $ where $p$ is a permutation of  $x,y,z \in X$ (symmetry in all three variables),
   \item[(G5)]
   $G(x,y,z)\leq G(x,a,a) + G(a,y,z)$ for all  $x,y,z,a \in X$ (rectangle inequality).
 \end{itemize}
Then  the function  $G$  is called a {\it G-metric} on X and the pair $(X,G)$ is called  a {\it G-metric space}.
 Let $ \{x_{n} \}_{n\geq 1}$ be a sequence in  X.
$ \{x_{n} \}_{n\geq 1}$  is called a {\it G-convergent sequence to $x \in X$} if, for any $\varepsilon > 0$, there is  $k\in \mathbb{N}$ such that for all
$n,m \geq k$, $G(x, x_{n}, x_{m}) <\varepsilon$ \cite{Zead}.
   $ \{x_{n} \}_{n\geq 1}$  is a  {\it G-Cauchy sequence  } if, for any $\varepsilon > 0$, there is  $k\in \mathbb{N}$
such that for all $ n,m, l \geq k$, $G(x_{n}, x_{m}, x_{l}) < \varepsilon$ \cite{Zead}.
$(X,G)$ is said to be {\it G-complete } if every G-Cauchy sequence in X is  G-convergent
in $X$ \cite{Zead}.
\begin{prop}
\text{(see  \cite{Zead})}\label{1.2}  Let $(X,G)$ be a G-metric space and let $\{x_{n}\}_{n\geq 1}$ be a sequence in X. Then the following are equivalent:
\begin{enumerate}
\item[$(1)$]
the sequence $\{x_{n}\}_{n\geq 1}$ is  G-Cauchy,
\item[$(2)$]
$G(x_{n}, x_{m}, x_{m})\rightarrow 0$ as $n, m\rightarrow \infty$.
\end{enumerate}
\end{prop}

\begin{defin}
\text{(see \cite {Zead})} A G-metric space $(X,G)$ is said  to be {\it symmetric } if
\begin{equation}
G(x,y,y)=G(y,x,x) \text{ for all}~ x,y \in X.
\end{equation}
\end{defin}
\begin{example}
 Let $X=S^\omega$ be  the  set of $\omega$-sequences over a non-empty set $S$. The function  $G:X \times X
 \times X\longrightarrow  [0, \infty )$, defined by $$G(x,y,z)=max \{ 2^{-sup\{i|i\in \omega\forall j<i, x_{j}=y_{j}\}},  2^{-sup\{i|i\in \omega\forall j<i, y_{j}=z_{j}\}},  2^{-sup\{i|i\in \omega\forall j<i, x_{j}=z_{j}\}} \},$$
 for all $x,y,z \in X$, is a G- metric on X such that $(X,G)$   is symmetric. We note that $G$ is a generalization of the  well known {\it Baire metric},   see \cite{ kahn,9} for more details.
\end{example}
 \begin{example}(see \cite{goah})\label{ghor}
Let $X=[0,\infty)$ and
\begin{equation}
G(x,y,z)=max \{x,y\}+ max\{x,z\}+ max \{y,z\} -x-y-z\textsl{ for all} ~x,y,z \in X.
\end{equation}
 Then, $(X,G)$ is a G-metric space  and it  is symmetric.
 \end{example}
Example \ref{ghor} leads  us to the following result.

\begin{prop}\label{bagher} Let  $(X,p)$ be  a  partial metric space  and
\begin{equation}\label{1.3}
  G_{p}(x,y,z)= p(x,y)+ p(y,z)+ p(x,z) -p(x,x)-p(y,y)-p(z,z) \text{ for all}  ~ x,y,z \in X,
\end{equation}
then $(X,G_{P})$ is a G-metric space which is symmetric.
\end{prop}
\begin{proof}
The arguments are straightforward. As an illustration we prove (G2). If $G_{p}(x,y,y)=0$, then  $G_{p}(x,y,y)=p^{s}(x,y)=2p(x,y)-p(x,x)-p(y,y)=0$. Thus, $x=y$ since $(X,p^{s})$ is a metric space.
\end{proof}
\begin{lemma}\label{1mon}
 Let $(X, p)$ be a partial metric space. Then the following statements hold.
 \begin{itemize}
   \item[(A)]
 A sequence $\{x_{n}\}_{n\geq 1}$ in  $(X, p)$ is Cauchy if and only if it is G-Cauchy in its  associated G-metric space  $(X, G_{p})$.

   \item[(B)]  $(X, p)$ is complete if and only  if its  associated G-metric space $(X, G_{p})$ is G-complete. Moreover,
        $  {\rm lim}_{n, m \rightarrow \infty} G_{p}(x, x_{n}, x_{m}) = 0 \Leftrightarrow  p( x, x) = {\rm lim}_{n, m\rightarrow \infty} p( x_{n},x_{m}) = \\{\rm lim}_{n\rightarrow \infty} p(x_{n}, x_{n})={\rm lim}_{n \rightarrow \infty }p(x,x_{n})$.

\end{itemize}
 \end{lemma}
 \begin{proof} (A)
  Let $\{x_{n}\}_{n\geq 1}$ be a Cauchy sequence  in  $(X,p)$. Then by Propositions  \ref{1.2} and \ref{bagher}, it is easily seen that $\{x_{n}\}_{n\geq 1}$ is a G-Cauchy sequence in the G-metric space $(X,G_p)$. Conversely, suppose that $\{x_{n}\}_{n\geq 1}$ is a G-Cauchy sequence in $(X,G_p)$ and $\varepsilon > 0$ is given. Then there exists $n_{0}\in \mathbb{N}$
such that $\frac{1}{4}\varepsilon  > G_{p}(x_{m}, x_{m}, x_{n}) $ for all $n,m \geq n_{0}$, and so
\begin{eqnarray*}
G_{p}(x_{n}, x_{n}, x_{n_{0}}) + p(x_{n_0},x_{n_0}) = 2p(x_{n},x_{n_0})-p(x_{n},x_{n})\geq p(x_{n},x_{n}),
\end{eqnarray*}
 by  (P2). Consequently, the sequence $\{p(x_{n}, x_{n})\}$ is bounded in $\mathbb{R}$, and so there
exists a real number, say $a$, such that a subsequence $\{p(x_{n_{k}} ,x_{n_{k}})\}$ is convergent to
$a$, i.e., \\ ${\rm lim}_{k\rightarrow \infty } p(x_{n_{k}} ,x_{n_{k}}) = a$. First we claim  that $\{p(x_{n}, x_{n})\}$ is a Cauchy sequence in $\mathbb{R}$.
If
$n,m \geq  n_{0}$, then
\begin{equation}
G_{p}(x_{m}, x_{m}, x_{n})=  2p(x_{m},  x_{n})- p(x_{m}, x_{m})- p(x_{n}, x_{n}) \nonumber
\end{equation}
\begin{equation}\label{22feb}
\geq p(x_{m}, x_{m})-p(x_{n}, x_{n}) , \nonumber
 \end{equation}
 by    (P2), which implies that
\begin{eqnarray*}
|p(x_{n}, x_{n}) -p(x_{m}, x_{m})|\leq G_{p}(x_{m}, x_{m}, x_{n})+G_{P}(x_{n}, x_{n}, x_{m}).
\end{eqnarray*}
Thus, $ | p(x_{n}, x_{n}) - p(x_{m}, x_{m})|< \frac{\varepsilon}{2}$ for all $n,m \geq n_{0}$, and so  ${\rm lim}_{n\rightarrow \infty} p(x_{n}, x_{n}) = a$. Thus, for any $\delta >0$ there exists $n_{\delta}$ such that $ |p(x_{k}, x_{k}) - a |< \delta$ and  $ G_{p}(x_{n}, x_{m}, x_{m})< \delta$ for all $k \geq  n_{\delta}$
 and so
\begin{eqnarray*}
2\vert p(x_{n}, x_{m}) - a\vert \\ =  \vert G_{p}(x_{n}, x_{m}, x_{m}) + p(x_{m}, x_{m})+ p(x_{n}, x_{n})
- 2a\vert \\  \leq G_{p}(x_{n}, x_{m}, x_{m}) + \vert  p(x_{m}, x_{m}) -a\vert  \\
 +  \vert p(x_{n}, x_{n}) - a \vert   <3 \delta
\end{eqnarray*}
for all $n,m \geq  n_{\delta}$. Hence,  ${\rm lim}_{n,m\rightarrow \infty } p(x_{n}, x_{m}) = a$ and so $\{x_{n}\}_{n\geq 1}$ is a
Cauchy sequence in $(X, p)$. \par
(B) Every G-metric space $(X,G)$ which is symmetric induces a metric $d_G$ on X given by
\begin{eqnarray}
  d_{G}(x,y)=G(x,y,y) \text{ for all}  ~ x,y \in X. \nonumber
\end{eqnarray}
Moreover, $(X,G)$ is G-complete if and only if $(X,d_{G})$ is complete. It is easily seen that $d_{G_{P}}(x,y)=p^{s}(x,y)
$ for all $x,y \in X$. By Lemma \ref{1},   $(X,p)$ is complete if and only if the metric space $(X,p^{s})$ is complete, and so the G-metric space $(X,G_{p})$
is G-complete if and only if  $(X,p)$ is  complete. By  the equality (\ref{1.3})  and (P2),  ${\rm lim}_{n, m \rightarrow \infty} G_{p}(x, x_{n}, x_{m}) = 0$  if and only if $p(x,x)={\rm lim}_{n \rightarrow \infty }p(x,x_{n})$, ${\rm lim}_{n, m \rightarrow \infty} p(x_{n},x_{m})=p(x_{n},x_{n})$ and ${\rm lim}_{m\rightarrow \infty }p(x,x_{m})=p(x_{m},x_{m})$ and these are equivalent to $$p( x, x) = {\rm lim}_{n, m\rightarrow \infty} p( x_{n},x_{m}) = {\rm lim}_{n\rightarrow \infty} p(x_{n}, x_{n})={\rm lim}_{n \rightarrow \infty }p(x,x_{n}).$$
\end{proof}
\begin{remark}\label{1shamb}

Haghi $et~al.$ \cite{hagh1,hagh2} showed that some of fixed point theorems in partial metric spaces can be obtained from metric spaces. From Proposition \ref{bagher} and Lemma \ref{1mon},  it follows  that  some  new fixed point results on partial metric  spaces are immediate consequences of the existing fixed point theorems on G-metric spaces and so they are not real generalizations. As a model example, we consider  the following result of Karapinar and  Agarwal  \cite{a}.
\end{remark}
\begin{thm}\text{ (see \cite{a})}\label{mus}
Let $(X,G)$ be a  G-metric space. If $(X,G)$ is G-complete and   $T:X\to X$ is  a mapping satisfying
\begin{equation}
G(Tx,T^{2}x,Ty)\leq G(x,Tx,y)- \phi(G(x,Tx,y))~~ \text{for all}~ x,y\in X,
\end{equation}
where $\phi :[0,\infty)\to [0,\infty)$ is a continuous mapping with $\phi^{-1}(\{0\})=\{0\}$. Then $T$ has a unique fixed point.
\end{thm}
We show that the following  fixed point theorem  on partial metric  spaces  is obtained by the above results.

\begin{thm}
Let $(X,p)$ be a complete partial metric space. If $T:X\to X$  is  a mapping and  the following condition holds, then $T$ has a unique fixed point.
\\ $p(Tx,T^{2}x)+p(Tx,Ty)+ p(T^{2}x,Ty)-p(T^{2}x,T^{2}x)-p(Ty,Ty)\leq p(x,Tx)+ p(x,y)+$ \\ $p(Tx,y)-p(x,x)- p(y,y)- \phi(p(x,Tx)+ p(x,y)+ p(Tx,y)-p(x,x)- p(y,y)-p(Tx,Tx))$ for all $x,y\in X$,
where $\phi :[0,\infty)\to [0,\infty)$ is a continuous mapping with $\phi^{-1}(\{0\})=\{0\}$.
\end{thm}
\begin{proof}
 Let $(X,G_{p})$  be the G-metric space defined in Proposition \ref{bagher}.
     Since $(X,p)$ is  complete, then by Lemma \ref{1mon},     $(X,G_{p})$ is G-complete.  Therefore, $T$ has a unique fixed point  since all conditions of  Theorem \ref{mus}  hold for the G-metric space $(X,G_{p})$.
  \end{proof}
\section{Remarks on GP-metric spaces and fixed points theorems}
Zand and Nezhad \cite{M.} recently
introduced GP-metric  spaces which are a generalization  of the notions of partial metric spaces, and G-metric spaces which are symmetric.  In this section,  we show that every  fixed point theorem on  a partial metric space implies a   fixed point result  on its associated  GP-metric space.  Following Zand and Nezhad \cite{M.}, the notion of GP-metric space is given as follows. \par
 Let $X$ be a non-empty  set and  $GP:X \times X
 \times X\longrightarrow [0, \infty )$  be a function satisfying the following conditions:
 \begin{itemize}
 \item[(GP1)]
 $0\leq GP(x,x,x)\leq GP(x,x,y)\leq GP(x,y,z)$ for all  $x,y,z\in X$,
 \item[(GP2)]
  $GP(x,y,z)=GP( p \{x,z,y\}) $ where $p$ is a permutation of  $x,y,z \in X$ (symmetry in all three variables),
 \item[(GP3)]
  $ GP(x,y,z)\leq GP(x,a,a) + GP(a,y,z) - GP(a,a,a)$ for all  $x,y,z,a\in X$ (rectangle inequality),
 \item[(GP4)]
 $x=y=z$ if $GP(x,y,z)=GP(x,x,x)=GP(y,y,y)=GP(z,z,z)$ for all  $x,y,z\in X$.
  \end{itemize}
  Then the function   $GP$  is   called a {\it GP-metric} on X and  the pair  $(X,GP)$  is called a {\it GP-metric space}. By (GP1) and (GP2), every GP-metric space  $(X,GP)$ is symmetric, that is,
   \begin{equation}
GP(x,y,y)=GP(y,x,x) \text{ for all}~ x,y \in X.
\end{equation}
An {\it open ball} for    $(X, GP)$ is a set of the form,
\begin{eqnarray}
\beta ^{GP} (x_{0}, \varepsilon) = \{ y\in X~\vert ~ GP(x_{0}, y, y) <\varepsilon + GP(x_{0},x_{0}, x_{0})\}\nonumber
\end{eqnarray}
for each  $x\in X$ and  $\varepsilon >0$. Also, these open balls form a base for  a topology $\tau_{GP}$ on X. \par   Let $(X, GP)$  and $(Y, GP1)$ be two  GP-metric spaces.  A function $f: (X, GP) \to (Y, GP1)$  is said to be {\it GP- continuous} if for any $\varepsilon >0$  and $x_{0}\in X$ there exists $\delta >0$ such that $f(\beta ^{GP} (x_{0}, \delta)) \subseteq \beta ^{GP1} (f(x_{0}) ,\varepsilon)$ \cite{M.}.

  \begin{lemma}\text{(see  \cite{ps})} \label{karp}
Let $(X,GP)$ be a GP-metric space. Then $GP(x,x,y)>0$ for all $x,y\in X$ with $x\neq y$.
\end{lemma}
\begin{defin}\text{(see  \cite{M.})}
 Let $(X, GP)$  be a  GP-metric space and let $ \{x_{n} \}_{n\geq 1}$ be a  sequence in $X$.
\begin{itemize}
  \item [(i)]
  $ ~~~~~~\{x_{n} \}_{n\geq 1}$ is  called  a  {\it GP-Cauchy sequence}    if ${\rm lim}_{n, m\rightarrow \infty}GP(x_{n}, x_{m}, x_{m})$ exists (and finite).
   \item [(ii)]
    $(X, GP)$ is said to be  {\it GP-complete}   if every GP-Cauchy sequence $\{y_{n}\}_{n\geq 1}$ converges to a point $x\in X$ such that $GP(x, x, x)= {\rm lim}_{n, m\rightarrow \infty} GP(y_{n}, y_{m}, y_{m})$.
    \end{itemize}
  \end{defin}
 \begin{lemma}\text{(see \cite{M.})} \label{karp}
Let $(X,GP)$ be a GP-metric space and $ \{x_{n} \}_{n\geq 1}$ GP-convergent sequence  to $x\in X
$. Then $ \{x_{n} \}_{n\geq 1}$ is a GP-Cauchy sequence and ${\rm lim }_{n, m \rightarrow \infty }GP(x, x_{n}, x_{m}) =GP(x, x, x)$.
\end{lemma}
We will use the following definition of  a GP-convergent sequence in a GP-metric space.
\begin{defin}\text{ (see \cite{goah})}
A sequence $ \{x_{n} \}_{n\geq 1}$ in a  GP-metric space $(X,GP)$ is called a {\it GP-convergent sequence  to } $x\in X$,  if

    $ GP(x, x, x)={\rm lim }_{n \rightarrow \infty }GP(x_{n}, x_{n}, x_{n}) ={\rm lim }_{n \rightarrow \infty } GP(x, x, x_{n}).$
    \end{defin}


The following theorem is needed in this paper.
\begin{thm}\text{\cite{goah}}\label{di}
 Let $(X, GP)$ be a GP-metric space. Then, the  function $ GP (x,y,z)$  is  jointly continuous in all three of its variables.
\end{thm}

The proof of the following result is straightforward:
\begin{prop}\label{dosh2}
Every  GP-metric space $(X,GP)$ defines a partial metric space $(X,p_{GP})$, where
\begin{equation}\label{baghery}
 p_{GP}(x,y)=GP(x,y,y)  \text{ for all}  ~ x,y \in X.
\end{equation}
\end{prop}
The following result is a direct consequence of Proposition \ref{dosh2} and the above definitions.
\begin{corollary}\label{2sh}

\begin{itemize}
 \item[(1)] A sequence $ \{x_{n} \}_{n\geq 1}$  is a GP-Cauchy sequence in a GP-metric space $(X,GP)$ if and only if it is a Cauchy sequence  in its associated partial metric space $(X,p_{GP})$.\nonumber
 \item[(2)] A GP-metric space $(X,GP)$ is GP-complete if and only if its associated partial metric space   $(X,p_{GP})$ is complete.
\end{itemize}
\end{corollary}

By Propositions \ref{bagher} and \ref{dosh2}, we have the following result.
\begin{thm}\label{gomeh}
 Every  GP-metric space $(X,GP)$ defines a  G-metric space $(X,G_{GP})$ which is symmetric, where
\begin{itemize}\label{nafea}
  \item[]
  $G_{GP}(x,y,z)=GP(x,y,y)+GP(x,z,z)+GP(y,z,z)-GP(x,x,x)-GP(y,y,y)-GP(z,z,z) $ for all  $ x,y ,z\in X$.
\end{itemize}
\end{thm}

\begin{corollary}\label{29morddad}
Let $(X,GP)$ be a GP-metric space and let $(X,G_{GP})$ be its associated G-metric space. Then  the following conditions hold.
\begin{itemize}

  \item[(1)] A sequence is GP-Cauchy in  $(X,GP)$ if and only if it is a G-Cauchy sequence in  $(X,G_{GP})$.
    \item[(2)]  $(X,GP)$ is GP-complete  if and only if   $(X,G_{GP})$ is G-complete. Moreover, \\  $  {\rm lim}_{n, m \rightarrow \infty} G_{GP}(x, x_{n}, x_{m}) = 0 \Leftrightarrow  GP( x, x,x) =$ \\$ {\rm lim}_{n\rightarrow \infty} GP( x,x_{n},x_{n}) = {\rm lim}_{n,m\rightarrow \infty} GP(x_{n},x_{m}, x_{m}) = {\rm lim}_{n\rightarrow \infty} GP(x_{n},x_{n}, x_{n})$.

\end{itemize}
\end{corollary}

\begin{remark}\label{2}
From Proposition \ref{dosh2} and Corollary  \ref{2sh},  it follows  that a fixed point theorem on  partial metric spaces implies a fixed point theorem  on GP-metric spaces, and so by  Theorem \ref{gomeh} and Corollary \ref{29morddad}, from every  fixed point theorem  on G-metric spaces a fixed point result on GP-metric spaces can be obtained.
\end{remark}

For instance, let $(X,p)$ be a complete partial metric space, $\phi: X \to [0,\infty) $ be a lower semi continuous function and $T$ a self-mapping on $X$ such
that 
$p(x,Tx) \leq \phi(x) - \phi(Tx)~ \textsl{for all} ~x\in X.
$
 ~~In 2011, Karapinar proved that the self-mapping  $T$  called a Caristi map on $(X,p)$ has a fixed point  \cite{aaa}. Thus by the Karapinar result, the following fixed point  Theorem on  GP-metric spaces  can be obtained.
\begin{thm}\label{crits}
Let $(X,GP)$ be a GP-metric space, $\phi: X \to [0,\infty) $ be a lower semi continuous function and $T$ a self-mapping on $X$ such
that
\begin{equation}
GP(x,Tx,Tx) \leq \phi(x) - \phi(Tx)~ \textsl{for all} ~x\in X.\nonumber
\end{equation}
If $(X,GP)$ is GP-complete, then the self-mapping $T$ has a fixed point.
\end{thm}
\begin{proof}
By Proposition \ref{dosh2} and Corollary \ref{2sh}, $(X,P_{GP})$ is a  complete partial metric space, where
\begin{equation}\label{baghery}
 p_{GP}(x,y)=GP(x,y,y)  \text{ for all}  ~ x,y \in X.
\end{equation}
Thus by hypothesis
\begin{equation}
p_{GP}(x,Tx) \leq \phi(x) - \phi(Tx)~ \textsl{for all} ~x\in X.\nonumber
\end{equation}
The Karapinar result mentioned above completes the proof.
\end{proof}

Finally, we present a new fixed point theorem on  GP-metric spaces such that this theorem can not be obtained from the well known fixed point results by the above techniques.
 We will  start with the following definition.
\begin{defin}
 Let $(X, GP)$ be  a GP-metric space and let $\phi : X \times X \rightarrow [0, \infty ) $  be a map. If    $T:(X,GP)\to (X,GP)$ is a GP-continuous function and $\{ x_{n} \}$ is a GP-convergent sequence to $x\in X$
 such that the following implication holds,
 \begin{eqnarray*}
 {\rm lim}_{n, m\rightarrow \infty }  GP  (x_{n}, x_{m}, Tx_{m}) = GP(x, x, x) \Rightarrow \\ \phi (x, Tx)
 \leq {\rm lim}_{ m\rightarrow \infty} {\rm inf} \phi (x_{m}, Tx_{m})=sup_{n\geq 1}inf_{m\geq n}\phi (x_{m}, Tx_{m}),
 \end{eqnarray*}
 then  the function $\phi$ is called a {\it T-lower semi-continuous} (T-l.s.c.) on $X\times X$.
Moreover, If for all $x\in X$ the following inequality holds:
   \begin{equation}\label{3.1}
GP(x,Tx,T^{2}x)\leq \phi (x,Tx)-\phi (Tx,T^{2}x),
\end{equation}
then $T$ is called a {\it  Caristi map} on $(X, GP)$.
 \end{defin}
 Now we give a generalization of Theorem \ref{crits}.

\begin{thm}\label{thm}
Let $\phi : X \times X \rightarrow [0, \infty )$, $ T:X\to X$ be two maps  and let $(X, GP)$ be a GP-metric space which is  GP-complete.  If   $\phi $ is  a T-l.s.c. function on $X\times X$  and $T$ is  a  Caristi map on $(X, GP)$. Then,    $T$  has a fixed point in $X$.
\end{thm}
\begin{proof}
For   $x$    in  $X$, let  $A(x)$ and  $a (x)$ denote  the following sets:
\begin{eqnarray}\label{4}
A(x)&=&\{ z \in X~\vert ~ GP(x, z, Tz)\leq \phi (x, Tx)-\phi (z, Tz)\}, \\ 
a(x) &=&{\rm inf} \{ \phi (z, Tz)~\vert ~ z\in A(x)\}.
\end{eqnarray}
By (\ref{3.1}),  $Tx\in A(x)$ and so  $A(x)\neq \varnothing$.   $z\in A(x)$ implies that $0\leq \alpha (x)\leq \phi (z, Tz)\leq \phi (x, Tx)$.  A  sequence $\{x_{n}\}_{n\geq 1}$ can be constructed  in the following way:
\begin{eqnarray}
x_{1}& := &x,\nonumber \\
x_{n+1}& \in & A(x_{n})~ \textsl{such that}~ \phi (x_{n+1}, Tx_{n+1})\leq a(x_{n}) +\frac{1}{n}, \forall n\in \mathbb{N}.
\end{eqnarray}
For every $n \in \mathbb{N}$, one can easily observe that
\begin{eqnarray}\label{10}
GP(x_{n}, x_{n+1}, Tx_{n+1})& \leq & \phi (x_{n}, Tx_{n})-\phi (x_{n+1}, Tx_{n+1}), \\
a(x_{n})& \leq &\phi (x_{n+1}, Tx_{n+1})  \leq   a(x_{n}) + \frac{1}{n},\nonumber
\end{eqnarray}
  and so  $\{\phi (x_{n}, Tx_{n})\}$ is a decreasing sequence of non-negative real numbers Thus, the sequence $\{\phi (x_{n}, Tx_{n})\}$ is convergent to a   real number $L$, and so by  (\ref{10}), we have
\begin{equation}\label{11}
L= {\rm lim }_{n\rightarrow \infty}\phi (x_{n}, Tx_{n}) = {\rm lim }_{n\rightarrow \infty} a(x_{n}).
\end{equation} By (\ref{11}), for each $k\in \mathbb{N}$ there exists $N_{k}\in \mathbb{N}$ such that
\begin{equation}
\phi (x_{n}, Tx_{n}) \leq L + \frac{1}{k}, \textsl{ for all } n \geq N_{k}.
\end{equation}
Since  $\{\phi (x_{n}, Tx_{n})\}$ is a decreasing sequence,  we have
\begin{equation}
L\leq \phi (x_{m}, Tx_{m})\leq \phi (x_{n}, Tx_{n}) \leq L+\frac{1}{k},  ~\textsl{for all} ~ m\geq n\geq N_{k},\nonumber
\end{equation}
and so
\begin{equation}\label{14}
\phi (x_{n}, Tx_{n}) - \phi (x_{m}, Tx_{m}) < \frac{1}{k}, ~\textsl{for all} ~ m\geq n\geq N_{k}.
\end{equation}
By (GP3), (GP1) and  (\ref{10}), the following inequalities hold.
\begin{eqnarray}\label{12}
GP(x_{n}, x_{n+2}, Tx_{n+2})& \leq & GP(x_{n}, x_{n+1}, x_{n+1}) + GP(x_{n+1}, x_{n+2}, Tx_{n+2}) -GP(x_{n+1}, x_{n+1}, x_{n+1}) \nonumber \\
& \leq & GP(x_{n}, x_{n+1}, Tx_{n+1}) + GP(x_{n+1}, x_{n+2}, Tx_{n+2}) \nonumber \\
& \leq & \phi (x_{n}, Tx_{n}) - \phi (x_{n+1}, Tx_{n+1}) +\phi (x_{n+1}, Tx_{n+1}) - \phi (x_{n+2}, Tx_{n+2})\nonumber \\
& = & \phi (x_{n}, Tx_{n}) - \phi (x_{n+2}, Tx_{n+2})  \nonumber .
\end{eqnarray}
Similar to the above calculations we have the following inequalities.
\begin{eqnarray}\label{13}
GP(x_{n}, x_{n+3}, Tx_{n+3})& \leq & GP(x_{n}, x_{n+2}, x_{n+2}) + GP(x_{n+2}, x_{n+3}, Tx_{n+3})
- GP(x_{n+2}, x_{n+2}, x_{n+2})  \nonumber\\
& \leq & GP(x_{n}, x_{n+2},  Tx_{n+2}) +  GP(x_{n+2}, x_{n+3}, Tx_{n+3})  \nonumber\\
 & \leq & \phi (x_{n}, Tx_{n}) - \phi (x_{n+2}, Tx_{n+2}) + \phi (x_{n+2}, Tx_{n+2})  \nonumber \\
& - &  \phi (x_{n+3}, Tx_{n+3}) \nonumber \\
& =  & \phi (x_{n}, Tx_{n}) - \phi (x_{n+3}, Tx_{n+3}). \nonumber
\end{eqnarray}
By  a simple  induction, the following inequality holds.
\begin{equation}\label{15}
GP(x_{n}, x_{m}, Tx_{m}) \leq \phi (x_{n}, Tx_{n}) - \phi (x_{m}, Tx_{m}) ~\textsl{for all} ~m\geq n,
\end{equation}
By   Theorem \ref{nafea} and  (GP1), we have:
\begin{eqnarray}\label{17}
G_{GP}(x_{n}, x_{m}, x_{m}) &\leq & 2GP(x_{n}, x_{m}, x_{m})  \nonumber \\
& \leq & 2GP(x_{n}, x_{m}, Tx_{m}).
\end{eqnarray}
  Since the sequence $\{ \phi (x_{n}, Tx_{n})\}$ is convergent, the right-hand side of (\ref{15})  tends to zero, and so  $GP(x_{n}, x_{m}, Tx_{m})$ tends to zero as $n, m \rightarrow \infty$. Thus  by (\ref{17}),   $\{x_{n}\}_{n\geq 1}$ is G-Cauchy in the G-metric space $(X,G_{GP})$. Since $(X, GP)$ is GP-complete, then by Corollary  \ref{29morddad} its associated G-metric space $(X, G_{GP})$ is G-complete. Thus there is $z_0\in X$ such that  the sequence $\{x_{n}\}_{n\geq 1}$ is G-convergent to $z_0 \in X$, and so by Corollary  \ref{29morddad}
\begin{eqnarray}\label{18}
GP(z_{0}, z_{0}, z_{0})  & =  & {\rm lim}_{n \rightarrow \infty} GP(z_{0}, x_{n}, x_{n})
  =  {\rm lim}_{n, m \rightarrow \infty }GP(x_{n}, x_{m}, x_{m}) \nonumber \\
 & \leq &  {\rm lim }_{n, m \rightarrow \infty}GP(x_{n}, x_{m}, Tx_{m})=0.
 \end{eqnarray}
  Thus,  $\{x_{n}\}_{n\geq 1}$ is GP-convergent to $z_0$ and $GP(z_{0}, z_{0}, z_{0})= {\rm lim }_{n, m \rightarrow \infty}GP(x_{n}, x_{m}, Tx_{m})$ and so by (\ref{11}), $\phi(z_{0}, Tz_{0})  \leq {\rm lim}_{m\rightarrow \infty} {\rm inf}\phi (x_{m}, Tx_{m})=L$ since  $\phi$ is T-l.s.c.. Therefore,  the inequality (\ref{15}) and Theorem \ref{di} imply that
\begin{eqnarray}
\phi(z_{0}, Tz_{0}) &  \leq &{\rm lim}_{m\rightarrow \infty} {\rm inf}\phi (x_{m}, Tx_{m}) \nonumber \\
& \leq & {\rm lim}_{m\rightarrow \infty} {\rm inf}[\phi (x_{n}, Tx_{n}) - G_{P}(x_{n}, x_{m}, Tx_{m})] \nonumber \\
& = & \phi (x_{n}, Tx_{n}) - GP(x_{n}, z_{0}, Tz_{0}).
\end{eqnarray}
 Therefore, (\ref{4}) implies that $z_{0}\in  A(x_{n})$  for all $n\in \mathbb{N}$ and so by   (\ref{11}),  $L= {\rm lim}_{n\rightarrow \infty} a(x_{n}) \leq   \phi (z_{0}, Tz_{0})$. Thus, $L = \phi (z_{0}, Tz_{0}).$ \\
By (GP3), (GP1) and  (\ref{15}) the following inequalities  hold.
\begin{eqnarray}\label{3.13}
0\leq GP(x_{n}, Tz_{0}, T^{2}z_{0}) & \leq & GP(x_{n}, z_{0}, z_{0}) + GP(z_{0}, Tz_{0}, T^{2}z_{0}) - GP(z_{0}, z_{0}, z_{0}) \nonumber \\
& \leq & GP(x_{n}, z_{0}, Tz_{0}) + GP(z_{0}, Tz_{0}, T^{2}z_{0}) \nonumber \\
& \leq & \phi (x_{n}, Tx_{n}) - \phi (z_{0}, Tz_{0}) + \phi (z_{0}, Tz_{0}) - \phi (Tz_{0}, T^{2}z_{0}) \nonumber \\
& = & \phi (x_{n}, Tx_{n}) - \phi (Tz_{0}, T^{2}z_{0}).  
\end{eqnarray}
 Clearly , $ \phi (Tz_{0}, T^{2}z_{0}) \leq \phi (x_{n}, Tx_{n})$ for all $n\in \mathbb{N}$ and so $\phi (Tz_{0}, T^{2}z_{0}) \leq L$ by (\ref{11}). By (\ref{4}) and (\ref{3.13}), $Tz_{0}\in  A(x_{n})$ for all $n\in \mathbb{N}$  which implies  that $a(x_{n}) \leq \phi (Tz_{0}, T^{2}z_{0})$ for all $n\in \mathbb{N}$,
  and so   $L\leq \phi (Tz_{0}, T^{2}z_{0})$ by (\ref{11}).  Hence, $\phi (z_{0}, Tz_{0})= L=\phi (Tz_{0}, T^{2}z_{0})$. Since $T$ is a   Caristi map on $(X, GP)$, we have $GP(z_{0}, Tz_{0}, T^{2}z_{0}) =0$. Regarding Lemma \ref{karp}, $Tz_{0}=T^{2}z_{0}=z_{0}$, which  completes the proof.
\end{proof}
\begin{example}
Let $X=[0, \infty)$ and $GP(x, y, z)={\rm max}\{x, y, z\}$, then $(X, GP)$ is a GP-metric space \cite{M.}. Suppose $T:X\rightarrow X$ such that $Tx= \frac{x}{2}$ for all $x\in X$ and $\phi(t,s) :X\times X\rightarrow [0, \infty)$ such that $\phi (t, s)=2(t+s)$. Then,
\begin{eqnarray}
GP(x, Tx, T^{2}x) = {\rm max}\{x, \frac{x}{2}, \frac{x}{4}\}=x~~~\textsl{and }~~~\phi (x, Tx)-\phi (Tx, T^{2}x)= \frac{3x}{2}.\nonumber
\end{eqnarray}
Therefore by  Theorem \ref{thm},  $T$ has a fixed point. Notice that $x=0$ is the fixed point of $T$.
\end{example}


\end{document}